\documentclass{article}

\usepackage{epic}
\usepackage{eepic}
\usepackage{mathtools}
\usepackage{ mathrsfs }
\usepackage{amsmath}
\usepackage{amsthm}
\usepackage{amssymb}
\usepackage{enumerate}
\usepackage{tikz-cd}

\usepackage{xcolor}

\usepackage{url}

\newcommand{\C}{\mathbb C}
\newcommand{\Q}{\mathbb Q}
\newcommand{\R}{\mathbb R}
\newcommand{\Z}{\mathbb Z}

\newcommand{\A}{\mathbb A}

\newcommand{\Rea}{\operatorname{Re}}

\newcommand{\Aut}{\operatorname{Aut}}

\newcommand{\tr}{\operatorname{tr}}

\newcommand{\GL}{\operatorname{GL}}

\newcommand{\Ort}{\operatorname{O}}

\newcommand{\Sort}{\operatorname{SO}}

\newcommand{\SL}{\mathrm{SL}}
\newcommand{\Sp}{\mathrm{Sp}}

\newcommand{\PGL}{\mathrm{PGL}}

\newcommand{\g}{g}

\newcommand{\comment}[1]{}
\newcommand{\St}{\operatorname{St}}

\newtheorem{thm}{Theorem}
\newtheorem*{con*}{Conjecture}
\newtheorem{prp}[thm]{Proposition}

\theoremstyle{definition}

\theoremstyle{remark}
\newtheorem{rem}{Remark}

\newcommand{\mainthm}{Let $m\geq24$ be an integer divisible by $8$. Then
	\[ \frac{m}{2}\leq \g_m\leq\frac{3m}{4}. \]}

\title{Linear independence of Theta series}
\author{Manuel Karl-Heinz M\"uller}

\begin{document}
	
	\begin{center}
		{\Large\bf Linear independence of theta series of positive-definite even unimodular lattices}\\[10mm]
		Manuel K.-H.\ M\"{u}ller,\\
		Sapienza Universit\`a di Roma, Rome, Italy, \\
		math@mkhmueller.de
	\end{center}
	
	\vspace*{1.2cm}
	\begin{center}
		\begin{minipage}{.8\textwidth}
			\noindent
			{\small We show that the minimal $\g$ for which the degree $\g$ theta series of positive-definite even unimodular lattices of rank $m\geq24$ are linearly independent is bounded between $\frac{m}{2}$ and $\frac{3m}{4}$.}
		\end{minipage}
	\end{center}
	
	\footnotetext{Mathematics Subject Classification 11F27 -- 11F46 -- 11F70}

	\vspace*{1.2cm}

	\section{Introduction}
	
	Let $L$ be a positive-definite even unimodular lattice of rank $m$ with bilinear form $(\cdot,\cdot)$. Then $m$ must be divisible by $8$. For an integer $\g\geq1$ we define the degree $\g$ \emph{theta series}
	\[ \theta_L^{(\g)}(\tau) \coloneqq \frac{1}{|\Aut(L)|}\sum_{\lambda\in L^\g}e^{\pi i\tr((\lambda,\lambda)\tau)}, \quad \tau\in\mathbb{H}_\g \]
	where $(\lambda,\lambda) = ((\lambda_i,\lambda_j))_{i,j=1}^\g\in\mathrm{Mat}_\g(\Z)$ for $\lambda = (\lambda_1,\hdots,\lambda_\g)\in L^\g$ and $\mathbb{H}_\g$ denotes the Siegel half-space of degree $\g$. By Poisson summation, $\theta_L^{(\g)}$ is a Siegel modular form of weight $m/2$ for the full modular group $\Sp_{2\g}(\Z)$. We denote the space of Siegel modular forms of weight $k$ by $\mathrm{M}_k(\Sp_{2\g}(\Z))$ and the subspace of cuspforms by $\mathrm{S}_k(\Sp_{2\g}(\Z))$. We may also define $\theta_L^{(0)} = \frac{1}{|\Aut(L)|}\in\C\eqqcolon\mathrm{M}_k(\Sp_{0}(\Z))$.
	
	Let $I\!I_{m,0}$ denote the genus of positive-definite even unimodular lattices of rank $m$ and let us write $[L]\in I\!I_{m,0}$ for the isomorphism class of $L$. Clearly, $\theta_L^{(\g)}$ only depends on $[L]$ so that it induces a homomorphism
	\[ \vartheta_m^{(\g)}: \C[I\!I_{m,0}] \to \mathrm{M}_{m/2}(\Sp_{2\g}(\Z)). \]
	The map $\vartheta_m^{(m)}$ is injective because if $[L],[M]\in I\!I_{m,0}$ and $G_M$ is any Gram matrix of $M$, then the Fourier coefficient of $\theta_L$ before $e^{\pi i\tr(G_M\tau)}$ is $\delta_{[L][M]}$. Let $\g_m\geq0$ be the smallest integer such that $\vartheta_m^{(\g_m)}$ is injective. For $m=8$ the genus $I\!I_{8,0}$ contains only the lattice $E_8$, hence, $\g_8 = 0$. In the case $m=16$ we have $\g_{16} = 4$, which comes from the existence of the \emph{Schottky form} in $\mathrm{S}_{8}(\Sp_{8}(\Z))$ \cite{Igusa}. It was shown in \cite{BFW} that $\g_{24} = 12$ by constructing a non-zero cusp form in $\mathrm{S}_{12}(\Sp_{24}(\Z))$ and this construction was generalized in \cite{Salvati} to all $m=24k$ yielding a lower bound $\g_{24k}\geq12k$. We will show
	
	
	\begin{thm}\label{thm:mainthm}
		\mainthm
	\end{thm}
	
	\begin{rem}\label{rem}
		Let $f\in\mathrm{S}_k(\SL_2(\Z))$ be an eigenform with automorphic $L$-function $L(s,f)$ (normalized such that $\frac{1}{2}$ is the natural center of the functional equation). When $k\equiv2\bmod4$, then $L\left(\frac{1}{2},f\right) = 0$. It is conjectured (see e.g.\ \cite{ConreyFarmer}) that $L\left(\frac{1}{2},f\right) \not= 0$ whenever $k\equiv0\mod4$. If this is the case for a given $k=\frac{m}{2}$, then $\g_{m}<\frac{3m}{4}$, otherwise $\g_{m} = \frac{3m}{4}$ (see end of Section \ref{sec:ProofOfTheorem}). The conjecture has been verified for all $k\leq2000$ in \cite{ConreyFarmer,FarmerJames}.
	\end{rem}
	\begin{rem}
		Theorem \ref{thm:mainthm} also implies that when $\g>\frac{3m}{4}$, then no cusp form $f\in\mathrm{S}_{m/2}(\Sp_{2\g}(\Z))$ is a linear combination of theta series $\theta_L^{(\g)}$ with $[L]\in I\!I_{m,0}$.
	\end{rem}
	\comment{\begin{rem}
			When $F\in\mathrm{S}_{\frac{k+\g}{2}}(\Sp_{2\g}(\Z))$ is the Ikeda lifting of an eigenform $f\in\mathrm{S}_k(\SL_2(\Z))$ with $k+\g\equiv0\bmod8$ such that $F$ lies in the image of $\vartheta_{k+\g}^{(\g)}$, then $\g\leq k$. This can be seen for example by comparing $L$-functions (\cite[Theorem 3.3]{Ikeda} and \cite[Folgerung 2.1]{Boecherer}). Hence, when $\g_m>m/2$, then the cuspforms in the image of $\vartheta_m^{(\g)}$ for $\g>m/2$ are not Ikeda liftings.
	\end{rem}}
	
	The space $\C[I\!I_{m,0}]$ can be identified with a certain space of automorphic forms for an orthogonal group. By a result of B\"ocherer, $\g_m$ can be determined from the zeros and poles of their corresponding $L$-functions. We will use Arthur's multiplicity formula to get information on the Arthur-Langlands parameters of the automorphic forms appearing in $\C[I\!I_{m,0}]$ from which their $L$-functions can be computed. The main arguments can be found in Section \ref{sec:ProofOfTheorem}. In Sections \ref{sec:AutomorphicRepresentatoins}, \ref{sec:MultiplicityFormula} and \ref{sec:ThetaCorrespondence} we introduce the necessary background on automorphic representations, the multiplicity formula and B\"ocherer's result on the theta correspondence. Our arguments are similar to the theory developed in \cite{CL}, where the interested reader may also find further background information on automorphic representations and multiplicity formulas for the groups this article concerns.
	\begin{rem}
		Note that the multiplicity formula of Arthur was conditional on the stabilization of the twisted trace formula for the groups $\GL_N$ and $\Sort_{2n}$ \cite[Hypothesis 3.2.1]{Arthur} as well as some results
		deferred to projected papers [A24]--[A27]. The issue corresponding to [A24] has since been resolved by M\oe glin--Waldspurger \cite{MoeglinWaldspurger1,MoeglinWaldspurger2} and in recent work by Atobe--Gan--Ichino--Kaletha--M\'inguez--Shin \cite{AGIKMS} the results of [A25]--[A27] have been established. These results leave Arthur's result conditional only on the twisted weighted fundamental lemma. Waldspurger has proved in \cite{Waldspurger} that this follows from the combination of the ordinary weighted fundamental lemma (WFL) and its non-standard version. Ordinary WFL for split groups was proved in \cite{ChaudouardLaumon1,ChaudouardLaumon2} so that the final missing pieces are ordinary WFL for non-split groups and non-standard WFL. This is Hypothesis (H1) in \cite{Shin}, where these dependencies are explained in more detail. The result of this article is also conditional on (H1). I would like to thank Ga\"etan Chenevier for originally pointing this out to me and the anonymous reviewer for providing more details on the current status.
	\end{rem}
	
	\medskip
	
	This work is funded by the Deutsche Forschungsgemeinschaft (DFG, German Research Foundation) -- 566117541.
	
	I would also like to thank R.\ Salvati Manni and Nils R.\ Scheithauer for stimulating discussions and helpful comments as well as the anonymous reviewer for carefully reading the manuscript and for suggesting several improvements.

	\section{Automorphic representations of orthogonal\\groups}\label{sec:AutomorphicRepresentatoins}
	
	Let $\Q_p$ denote the $p$-adic completion of $\Q$ at a prime $p$ with ring of integers $\Z_p$ and let $\A = \R\times\A_{\mathrm{f}} = \R\times\prod_{p<\infty}^\prime\Q_p$ be the adele ring of $\Q$ with $\widehat{\Z} = \prod_{p<\infty}\Z_p$.
	
	Throughout, $m$ will denote a positive integer divisible by $8$. We set $L_m = E_8^{\oplus m/8}$ and $V_m=L_m\otimes\R$ and define $\Z$-group schemes $\Sort_m$ and $\Ort_m$ via
	\[ R\mapsto \Sort(L_m\otimes R)\quad\text{and}\quad R\mapsto\Ort(L_m\otimes R), \]
	respectively for a ring $R$. Then $\Sort_m$ and $\Ort_m$ are reductive groups over $\Q$, where $\Sort_m$ is connected and $\Ort_m$ has two connected components. From now on let $G$ be either $\Sort_m$ or $\Ort_m$. Then the group $G(\A)$ has maximal compact subgroup $K_G=G(\R)\times G(\widehat{\Z})$. Let $\rho:G(\R)\to\GL(U)$ be an irreducible unitary representation. We denote by $\mathcal{A}_\rho(G)$ the space of $U$-valued functions $\phi:G(\A)\to U$ such that
	\[ \phi(\gamma xk) = \rho(k_\infty)^{-1}\phi(x) \]
	for $\gamma\in G(\Q)$, $x\in G(\A)$ and $k = (k_\infty,k_{\mathrm{f}})\in K_G$. Any function in $\mathcal{A}_\rho(G)$ is determined by the (finite) set of double cosets in $G(\Q)\backslash G(\A)/K_G$. Note that the map
	\[ \Ort_m(\Q)\backslash \Ort_m(\A)/K_{\Ort_m}\xrightarrow{\sim} I\!I_{m,0},\quad x\mapsto [x_{\mathrm{f}}(L_m\otimes\widehat{\Z})\cap L_m\otimes\Q] \]
	is a bijection. We say that two lattices $M,N\subset V_m$ are \emph{properly isomorphic} if there exists a $\gamma\in\Sort_m(\R)$ such that $N = \gamma(M)$ and define $\widetilde{I\!I}_{m,0}$ as the set of \emph{proper} isomorphism classes of positive-definite even unimodular lattices of rank $m$. Then similarly $\Sort_m(\Q)\backslash\Sort_m(\A)/K_{\Sort_m}\xrightarrow{\sim} \widetilde{I\!I}_{m,0}$ is bijective.
	
	The morphism $\Sort_m\hookrightarrow\Ort_m$ defines the natural surjective map $\widetilde{I\!I}_{m,0}\to I\!I_{m,0}$. The restriction of $\phi\in\mathcal{A}_\rho(\Ort_m)$ to $\Sort_m$ lies in $\mathcal{A}_{\rho'}(\Sort_m)$ where $\rho'=\rho|_{\Sort_m(\R)}$. This defines an inclusion
	\[ \mathcal{A}_\rho(\Ort_m)\hookrightarrow\mathcal{A}_{\rho'}(\Sort_m). \]
	We denote by $\C$ the trivial representation of $\Sort_m$ and $\Ort_m$, respectively. Then we can identify $\mathcal{A}_\C(\Sort_m)$ with $\C[\widetilde{I\!I}_{m,0}]$ and $\mathcal{A}_\C(\Ort_m)$ with $\C[I\!I_{m,0}]$ and the above inclusion induces an isomorphism
	\[ \mathcal{A}_\C(\Sort_m) \cong \mathcal{A}_\C(\Ort_m)\oplus\mathcal{A}_{\det}(\Ort_m) \]
	(see \cite[Section 4.4.4]{CL}).
	
	Let
	\[ \mathcal{H}_p(G) \coloneqq C_{\mathrm{c}}^\infty(G(\Z_p)\backslash G(\Q_p)/G(\Z_p)) \]
	be the \emph{local Hecke algebra} consisting of locally	constant, compactly supported functions on $G(\Q_p)$ that are left- and right-$G(\Z_p)$-invariant with the usual convolution product. Then 
	\[ \mathcal{H}(G) \coloneqq C_{\mathrm{c}}^\infty(G(\widehat{\Z})\backslash G(\A_{\mathrm{f}})/G(\widehat{\Z})) \cong \bigotimes_{p<\infty} \mathcal{H}_p(G) \]
	is called the \emph{global Hecke algebra}. It acts on $\mathcal{A}_\rho(G)$ from the right via
	\begin{align*}
		\mathcal{A}_\rho(G)\times \mathcal{H}(G)&\to\mathcal{A}_\rho(G) \\
		(\phi,f)&\mapsto(x\mapsto\int_{G(\A_{\mathrm{f}})} f(y)\phi(xy)\mathrm{d}y)
	\end{align*}
	and by transport of structure on $\C[\widetilde{I\!I}_{m,0}]$ (for $G=\Sort_m$) and $\C[I\!I_{m,0}]$ (for $G=\Ort_m$). The algebra $\mathcal{H}(\Ort_m)$ is commutative and hence, $\mathcal{A}_\rho(\Ort_m)$ is simultaneously diagonalizable with respect to $\mathcal{H}(\Ort_m)$.

	There exists a positive Radon measure on $G(\Q)\backslash G(\A)$. The space of square-integrable $\C$-valued functions $L^2(G(\Q)\backslash G(\A))$ is endowed with the right regular representation of $G(\A)$. An \emph{automorphic representation} of $G$ (in the $L^2$-sense) is an irreducible unitary representation of $G(\A)$ that is isomorphic to a subquotient of $L^2(G(\Q)\backslash G(\A))$ (cf.\ \cite[Definition 3.3]{GetzHahn}). We denote the set of isomorphism classes of automorphic representations by $\Pi(G)$. The subspace
	\[ L^2_{\mathrm{disc}}(G(\Q)\backslash G(\A))\subset L^2(G(\Q)\backslash G(\A)) \]
	which is the largest closed subspace that decomposes as a Hilbert space direct sum of irreducible subrepresentations of $L^2(G(\Q)\backslash G(\A))$ is called the \emph{discrete spectrum}. We denote by $\Pi_{\mathrm{disc}}(G)$ the subset of representations appearing in $L^2_{\mathrm{disc}}(G(\Q)\backslash G(\A))$ and let $m(\pi)$ denote the multiplicity of $\pi$ in $L^2_{\mathrm{disc}}(G(\Q)\backslash G(\A))$. Let $\mathrm{Irr}(G(\R))$ denote the set of isomorphism classes of topologically irreducible unitary	representations of $G(\R)$ and $\rho^*$ the dual representation of $\rho\in\mathrm{Irr}(G(\R))$. Then we have an isomorphism
	\begin{align*}
		\widehat{\bigoplus_{\rho\in\mathrm{Irr}(G(\R))}}\rho^*\otimes\mathcal{A}_\rho(G) &\xrightarrow{\sim} L^2_{\mathrm{disc}}(G(\Q)\backslash G(\A))^{G(\widehat{\Z})} \\
		f\otimes\phi&\mapsto (x\mapsto f(\phi(x)))
	\end{align*}
	where $L^2_{\mathrm{disc}}(G(\Q)\backslash G(\A))^{G(\widehat{\Z})}$ denotes the subspace of right-$G(\widehat{\Z})$-invariants (see \cite[Section 4.3.2]{CL}). Any $\pi\in\Pi_{\mathrm{disc}}(G)$ with $\dim\pi^{G(\widehat{\Z})}>0$ is of the form $\pi\cong\pi_\infty\otimes\pi_{\mathrm{f}}$ where $\pi_\infty\cong\rho^*$ for some $\rho\in\mathrm{Irr}(G(\R))$ and $\pi_{\mathrm{f}}$ is generated by a $\phi\in\mathcal{A}_\rho(G)$
	. We denote the subset of such representation $\Pi_{\mathrm{disc}}(G)^{G(\widehat{\Z})}$. In summary, we obtain
	\begin{thm}\label{thm:Fact}
		The space $\C[I\!I_{m,0}]$ has a basis consisting of eigenforms for $\mathcal{H}(\Ort_m)$. The eigenspaces are in one-to-one correspondence with the automorphic representations $\pi\in\Pi_{\mathrm{disc}}(\Ort_m)^{G(\widehat{\Z})}$ such that $\pi_\infty \cong\C$. The dimension of such an eigenspace is $m(\pi)$.
	\end{thm}
	
	\section{Arthur's multiplicity formula}\label{sec:MultiplicityFormula}
	
	We will now describe Arthur's multiplicity formula. In his original work \cite{Arthur}, Arthur considered quasi-split groups. The group $\Sort_m$ is not quasi-split, however, it is an inner twist of the (quasi-)split group $\Sort_{m/2,m/2}$ (the special orthogonal group of signature $(m/2,m/2)$) and Ta\"ibi extended the multiplicity formula to such groups for the case when the infinitesimal character at Archimedean places is algebraic regular in \cite{Taibi}. This suffices for our purposes. Note also \cite{ChenZou} for the treatment of the generic part of the automorphic discrete spectrum, but general infinitesimal characters.
	
	Let $G$ be a connected reductive group over $\Q$. Its \emph{Langlands dual group} $\widehat{G}$ is a reductive group over $\C$ with root datum dual to that of $G$. We do not want to go into more detail here since for us it suffices to know that the Langlands dual group of $\Sort_m$ is $\Sort_m(\C)$ (see e.g.\ \cite[Section 6.1.2]{CL}). When $G$ is a classical group, such as $\GL_n$, $\Sp_{2\g}$ or $\Sort_m$, then $\widehat{G}$ has a so-called \emph{standard representation} $\St$. In the case of $\Sort_m$, this is the inclusion
	\[ \St:\Sort_m(\C)\hookrightarrow\GL(L_m\otimes\C). \]
	Arthur associates to a representation $\pi\in\Pi_{\mathrm{disc}}(G)$ a formal symbol
	\[ \psi(\pi,\St) = \bigoplus_{i=1}^k\pi_i[d_i] \]
	called its \emph{standard parameter} where $\pi_i\in\Pi_{\mathrm{disc}}(\PGL_{n_i})$ are self-dual cuspidal automorphic representations, $d_i\in\Z_{\geq1}$ and $\sum_{i=1}^kn_id_i = m \coloneqq \dim(\St)$.
	Let $(c_\infty(\pi),c_2(\pi),c_3(\pi),\hdots)$ be the tuple indexed by $\infty$ and the primes where $c_p(\pi)\in\widehat{G}(\C)_{\mathrm{ss}}$ are the \emph{Satake-parameters} of $\pi$ and $c_\infty(\pi)\in\widehat{\mathfrak{g}}_{\mathrm{ss}}$ is the image of the infinitesimal character of $\pi_\infty$ under the Harish-Chandra isomorphism. Note that $\widehat{\mathfrak{g}}$ denotes the Lie-algebra of $\widehat{G}(\C)$ and $X_{\mathrm{ss}}$ denotes the set of conjugacy classes of semi-simple elements in $X$. We may also define $[d]_p\in\SL_d(\C)_{\mathrm{ss}}$ as the element with eigenvalues $p^{-\frac{d-1}{2}+j}$ for $j=0,\hdots,d-1$ and $[d]_\infty\in\mathrm{Mat}_d(\C)_{\mathrm{ss}}$ as the element with eigenvalues $-\frac{d-1}{2}+j$ for $j=0,\hdots,d-1$. Then for any formal symbol $\psi$ as above we set
	\[ \psi_p = \bigoplus_{i=1}^kc_p(\pi_i)\otimes[d_i]_p\in\SL_{m}(\C)_{\mathrm{ss}} \] 
	and equally for $p=\infty$. If $\psi$ is the standard parameter of $\pi$, then $\psi_\infty = \St(c_\infty(\pi))$ and $\psi_p = \St(c_p(\pi))$, respectively. The \emph{standard $L$-function} of $\pi$ is given by
	\begin{align*}
		L(s,\pi) &\coloneqq\prod_{p<\infty}\det(1-p^{-s}\St(c_p(\pi)))^{-1} \\
		&= \prod_{i=1}^k\prod_{j=0}^{d_i-1}L\left(s+j-\frac{d_i-1}{2},\pi_i\right)
	\end{align*}
	and the Euler product is absolutely convergent for $\Rea(s)>\frac{1+\max(d_i)_i}{2}$ following from the absolute convergence of $L(s,\pi_i)$ for $\Rea(s)>1$ (see \cite[Section 3]{Langlands}, \cite[Section 2.5]{Shahidi}, and \cite{JacquetShalika}). The functions $L(s,\pi_i)$ admit holomorphic extensions to all of $\C$ \cite{GodementJacquet} except when $\pi_i=1$ in which case $L(s,1) = \zeta(s)$ is the Riemann-$\zeta$-function, which is meromorphic on $\C$ with a simple pole at $s=1$. When $\pi\in\Pi_{\mathrm{disc}}(\PGL_{n})$ and $\pi'\in\Pi_{\mathrm{disc}}(\PGL_{n'})$ are self-dual and cuspidal, we also consider $L(s,\pi\times\pi')$. If $L_\infty(s,\pi\times,\pi')$ is a suitable product of \emph{$\Gamma$-factors}, then $\Lambda(s,\pi\times\pi') \coloneqq L_\infty(s,\pi\times\pi')L(s,\pi\times\pi')$ admits a functional equation of the form
	\[ \Lambda(s,\pi\times\pi') = \varepsilon(\pi\times\pi')\Lambda(1-s,\pi\times\pi') \]
	with $\varepsilon(\pi\times\pi')\in\C^\times$ (cf.\ \cite{Cogdell} for more information on $L$-functions).
	
	Finally, let 
	\[ \psi = \bigoplus_{i=1}^k\pi_i[d_i] \]
	and assume that the eigenvalues of $\psi_\infty$ are the integers
	\[ w_1>\hdots>w_{m/2}\geq-w_{m/2}>\hdots>-w_1. \]
	We define a character $\chi:\{\pm1\}^{I_0}\to\{\pm1\}$ where $I_0\subset\{1,\hdots,k\}$ is the subset consisting of those indices $i$ for which $n_id_i\equiv 0\bmod4$ in the following way: For $i\in\ I_0$ let $s_i\in\{\pm1\}^{I_0}$ be the element defined by $(s_i)_j = (-1)^{\delta_{ij}}$.
	\begin{enumerate}[(i)]
		\item If $d_i\equiv0\bmod2$, then $\chi(s_i) = (-1)^{n_id_i/4}$.
		\item If $d_i\equiv1\bmod2$, then $\chi(s_i) = (-1)^{|K_i|}$ where $K_i$ is the set of odd indices $1\leq j\leq m/2$ such that $w_j$ is an eigenvalue of $c_\infty(\pi_i)$.
	\end{enumerate}
	The following is Theorem 8.5.8 and Conjecture 8.1.2 in \cite{CL}, which were dependent on the more general case of Arthur's multiplicity formula that has since been proven by Ta\"ibi in \cite{Taibi}.
	\begin{thm}[Theorem 8.5.8 and Conjecture 8.1.2 in \cite{CL}, cf.\ \cite{Taibi}]\label{thm:Arthur}
		Let $\psi$, $I_0$ and $\chi$ be as above. Let $\Pi\subset\Pi_{\mathrm{disc}}(\Sort_m)^{\Sort_m(\widehat{\Z})}$ be the subset of representations such that $\psi(\pi,\St) = \psi$. Suppose that \cite[(H1)]{Shin} is true. Then $\Pi\not=\emptyset$ if and only if
		\[ \chi(s_i) = \prod_{1\leq j\leq k,j\not=i}\varepsilon(\pi_i\times\pi_j)^{\min(d_i,d_j)}\quad\forall i\in I_0. \]
		If this condition holds, then we have $\sum_{\pi\in\Pi}m(\pi) = 1$ if $I_0\not=\{1,\hdots,k\}$ and $\sum_{\pi\in\Pi}m(\pi) = 2$ otherwise. Furthermore, any $\pi\in\Pi_{\mathrm{disc}}(\Sort_m)^{\Sort_m(\widehat{\Z})}$ has a standard parameter of the form $\psi(\pi,\St) = \bigoplus_{i=1}^k\pi_i[d_i]$.
	\end{thm}
	
	\section{Theta correspondence}\label{sec:ThetaCorrespondence}
	
	We will now describe how $\g_m$ can be determined from the standard $L$-functions of the elements in $\Pi_\mathrm{disc}(\Ort_m)^{\Ort_m(\widehat{\Z})}$.	By Theorem \ref{thm:Fact}, any $\pi\in\Pi_\mathrm{disc}(\Ort_m)^{\Ort_m(\widehat{\Z})}$ with $\pi_\infty\cong\C$ defines a $m(\pi)$-dimensional eigenspace $V_\pi\subset\C[I\!I_{m,0}]$. We denote by $\g(\pi)$ the smallest integer $\g\geq0$ such that $\ker\vartheta_m^{(\g)}|_{V_\pi}=\{0\}$. The minimality of $\g$ is equivalent to requiring that $\vartheta_m^{(\g)}(V_\pi)\cap\mathrm{S}_{m/2}(\Sp_{2\g}(\Z))\not=\{0\}$ when $\g>0$. Let
	\[ T(\pi) \coloneqq\{t\in\Z_{\geq1}\mid L(s,\pi) \text{ has a zero or pole at } t\} \]
	and $t(\pi) = \max T(\pi)$. We then have the following theorem, due to B\"ocherer.
	\begin{thm}[Theorem 3.4 in \cite{Boecherer}]\label{thm:Boecherer}
		Let $\pi\in\Pi_{\mathrm{disc}}(\Ort_m)^{\Ort_m(\widehat{\Z})}$ with $\pi_\infty\cong\C$. Then
		\begin{enumerate}[(i)]
			\item If $T(\pi) = \emptyset$, then $\g(\pi) = \frac{m}{2}$.
			\item If $T(\pi) \not= \emptyset$ and $L(s,\pi)$ has a pole at $t(\pi)$, then $\g(\pi) = \frac{m}{2}-t(\pi)$.
			\item If $T(\pi) \not= \emptyset$ and $L(s,\pi)$ has a zero at $t(\pi)$, then $\g(\pi) = \frac{m}{2}+t(\pi)$.
		\end{enumerate}
	\end{thm}
	
	\section{Proof of the main theorem}\label{sec:ProofOfTheorem}
	
	Before we prove the main theorem let us make some observations: Let $\pi\in\Pi_\mathrm{disc}(\Ort_m)^{\Ort_m(\widehat{\Z})}$ with $\pi_\infty\cong\C$ and standard parameter $\psi = \bigoplus_{i=1}^k\pi_i[d_i]$ where $\pi_i\in\Pi_\mathrm{disc}(\PGL_{n_i})$.
	\begin{enumerate}[1)]
		\item Since the maps $\vartheta_m^{(\g)}$ are intertwining operators between the $\mathcal{H}(\Ort_m)$-module $\C[I\!I_{m,0}]$ and the $\mathcal{H}(\Sp_{2\g})$-modules $\mathrm{M}_{m/2}(\Sp_{2\g}(\Z))$ (see \cite{Rallis}), it follows that $\ker\vartheta_m^{(\g)}$ is an $\mathcal{H}(\Ort_m)$-submodule. Therefore, by Theorem \ref{thm:Fact}
		\[ \g_m = \max\{\g(\pi)\mid\pi\in\Pi_\mathrm{disc}(\Ort_m)^{\Ort_m(\widehat{\Z})}\text{ such that }\pi_\infty\cong\C\}. \]\label{obs:Max}
		\item The eigenvalues of $c_\infty(\pi)$ are the $m-2$ integers
		\[ \pm\left(\frac{m}{2}-1\right),\pm\left(\frac{m}{2}-2\right),\hdots,\pm1 \]
		as well as $0$ with multiplicity $2$ and $\C$ and $\det$ are the only representations with this infinitesimal character \cite[Section 6.4.3]{CL}.\label{obs:Eigenvalues}
		\item We have $d_i= \left(m-\sum_{j\not=i}n_jd_j\right)/n_i\leq m/n_i$ for all $i=1,\hdots,k$.\label{obs:BoundOnD}
		\item Either $n_id_i \equiv 0 \bmod 4$ for every $i$ or there exist exactly two integers $i$, say $i_1$ and $i_2$, such that $n_id_i \not\equiv 0 \bmod 4$. These integers satisfy $n_{i_1} d_{i_1}n_{i_2}d_{i_2} \equiv 3 \bmod 4$ \cite[Corollary 8.2.15 (iii)]{CL}.\label{obs:0mod4}
	\end{enumerate}
	
	\noindent The lower bound for $\g_m$ is obtained from
	
	\begin{prp}\label{prp:BorcherdsFreitagWeissauerForm}
		Let $f\in\mathrm{S}_{m/2}(\SL_2(\Z))$ be an eigenform with corresponding automorphic representation $\pi_f\in\Pi_{\mathrm{disc}}(\PGL_2)$. Then there exist $\pi^+,\pi^-\in\Pi_{\mathrm{disc}}(\Ort_m)^{\Ort_m(\widehat{\Z})}$ with $\pi^+_\infty\cong\C$, $\pi^-_\infty\cong\det$ and $\psi(\pi^\pm,\St) = \pi_f[\frac{m}{2}]$ and
		\[ \g(\pi^+) = \begin{cases}
			\frac{m}{2} & \text{if } L\left(\frac{1}{2},\pi_f\right)\not=0 \\
			\frac{3m}{4}& \text{if } L\left(\frac{1}{2},\pi_f\right)=0 .
		\end{cases} \]
	\end{prp}
	\begin{proof}
		Let the notation be as in Theorem \ref{thm:Arthur}. Since $k=1$ with $n_1=2$ and $d_1\equiv0\bmod4$, the character $\chi$ is trivial and clearly satisfies the condition of Theorem \ref{thm:Arthur}. Therefore, there exist $\pi',\pi''\in\Pi_{\mathrm{disc}}(\Sort_m)^{\Sort_m(\widehat{\Z})}$ with $\psi(\pi',\St) = \psi(\pi'',\St) = \pi_f[\frac{m}{2}]$. The eigenvalues of $\psi_\infty(\pi',\St)$ are
		\[ \pm\left(\frac{m/2-1}{2}\right)+\left(-\frac{m/2-1}{2}+j\right),\quad j=0,\hdots,\frac{m}{2}-1 \]
		so that by Observation \ref{obs:Eigenvalues}, $\pi'_\infty\cong\pi''_\infty\cong\C$.
		
		We need to show that there exist representations $\pi^\pm\in\Pi_{\mathrm{disc}}(\Ort_m)^{\Ort_m(\widehat{\Z})}$ for which $\pi^+\oplus\pi^-\cong\pi'\oplus\pi''$ when restricted to $\Sort_m$. This can be proved exactly as was done in \cite[Proposition 7.5.1]{CL} for $f=\Delta$. For the convenience of the reader we repeat the argument here. Note also \cite{AtobeGan}, where the passage from $\Sort_m$ to $\Ort_m$ and Arthur's multiplicity formula for quasi-split orthogonal groups is treated. Let $\pi'\cong\pi'_\infty\otimes\pi'_{\mathrm{f}}$ with $0\not=\phi\in\mathcal{A}_\C(\Sort_m)$ such that $\phi$ generates $\pi'_\mathrm{f}$. The non-trivial element $s\in\Ort_m(\Z)/\Sort_m(\Z)\cong\Z/2\Z$ acts on the space $\mathcal{A}_\C(\Sort_m)$ by conjugation. Under the decomposition
		\[ \mathcal{A}_\C(\Sort_m)\cong\mathcal{A}_\C(\Ort_m)\oplus\mathcal{A}_{\det}(\Ort_m) \]
		it acts trivially on $\mathcal{A}_\C(\Ort_m)$ and by multiplication by $-1$ on $\mathcal{A}_{\det}(\Ort_m)$. Let $s\pi'\in\Pi_{\mathrm{disc}}(\Sort_m)^{\Sort_m(\widehat{\Z})}$ be the representation generated by $s\phi$. We choose a prime $p$ and let $a(p)$ be the $p$-th Fourier coefficient of $f$. Then the conjugacy
		class $c_p(\pi') \subset \Sort_m(\C)$ does not have eigenvalue $\pm1$: The eigenvalues of $c_p(\pi')$ are equal to the eigenvalues of $c_p(\pi_f)$ times $p^{-m/4+1/2+j}$ for $j=0,\hdots,m/2-1$. The eigenvalues of $c_p(\pi_f)$ are given by the roots of $X^2 - (a(p)/p^{\frac{m/2-1}{2}})X + 1$, which are either $+1$, $-1$ or not real because $|a(p)|\leq 2p^{\frac{m/2-1}{2}}$ by the Ramanujan conjecture proved by Deligne. Therefore, the conjugacy	class $c_p(\pi')$ of $\Sort_m(\C)$ is not stable under the action of $\Ort_m(\C)$ by conjugation so that $c_p(\pi')\not=c_p(s\pi')$ as $\Sort_m(\C)$-conjugacy classes since the Satake isomorphism is compatible with the isomorphism induced by $s$. Hence, $\phi$ and $s\phi$ are not proportional. It follows that $\phi\pm s\phi$ generate representations $\pi^{\pm}\in\Pi_{\mathrm{disc}}(\Ort_m)^{\Ort_m(\widehat{\Z})}$ with $\pi^+_\infty\cong\C$, $\pi^-_\infty\cong\det$ and $\psi(\pi^\pm,\St) = \pi_f[\frac{m}{2}]$.
		
		To calculate $\g(\pi^+)$, we need to determine the zeros and poles of $L(s,\pi^+)$. Note that $\Lambda(s,\pi_f) = \Lambda(1-s,\pi_f)$ for $\Lambda(s,\pi_f) = \Gamma_\C\left(s+\frac{m/2-1}{2}\right)L(s,\pi_f)$ where $\Gamma_\C(s) = 2(2\pi)^{-s}\Gamma(s)$ (cf.\ \cite[Section 8.2.21]{CL}) so that the integer zeros of the factor $L\left(s+j-\frac{m/2-1}{2},\pi_f\right)$ lie at $s\in\Z_{\leq-j}$ and $s=m/4-j$ if $L\left(\frac{1}{2},\pi_f\right) = 0$. Since $L(s,\pi^+)$ is the product of these functions for $j=0,\hdots,m/2-1$, the claim follows from Theorem \ref{thm:Boecherer}.
	\end{proof}
	
	We can finally prove Theorem \ref{thm:mainthm}.
	
	\begin{proof}[Proof of Theorem \ref{thm:mainthm}]
		The lower bound follows from Proposition \ref{prp:BorcherdsFreitagWeissauerForm}. Now let $\pi\in\Pi_{\mathrm{disc}}(\Ort_m)^{\Ort_m(\widehat{\Z})}$ with $\pi_\infty = 1$. By Observation \ref{obs:Max} and Theorem \ref{thm:Boecherer}, we need to show that $L(t(\pi),\pi)$ is a pole or that $L(t,\pi)$ does not vanish for integers $t>m/4$. Let $\psi(\pi,\St) = \bigoplus_{i=1}^k\pi_i[d_i]$ and assume that $d_1 = \max(d_i)_i$. 
		
		Let us first assume that $d_1<m/2-1$. Let $t\geq m/4$. Then
		\[ t+j-\frac{d_i-1}{2} > \frac{m}{4}-\frac{m}{4}+1 = 1. \]
		Therefore,
		\[ L\left(t,\pi\right) = \prod_{i=1}^k\prod_{j=0}^{d_i-1}L\left(t+j-\frac{d_i-1}{2},\pi_i\right) \]
		does not vanish.
		
		Now assume that $d_1\geq m/2-1$. Then $n_1$ is either $1$ or $2$ by Observation \ref{obs:BoundOnD}. We first assume that $n_1=1$. Then by Observations \ref{obs:Eigenvalues} and \ref{obs:BoundOnD}, $d_i<d_1$ for $i>1$ and for $t\geq\frac{d_1+1}{2}$ we have
		\[ t + j - \frac{d_i-1}{2}>1 \]
		while for $t=\frac{d_1+1}{2}$ and $j=0$ we have
		\[ t + j - \frac{d_1-1}{2}=1. \]
		Hence, 
		\[\prod_{j=0}^{d_1-1}L\left(s+j-\frac{d_i-1}{2},1\right) \]
		has a simple pole at $\frac{d_1+1}{2}$ and 
		\[ \prod_{i=2}^k\prod_{j=0}^{d_i-1}L\left(t+j-\frac{d_i-1}{2},\pi_i\right) \]
		does not vanish for $t\geq\frac{d_1+1}{2}$. Therefore, $L(t(\pi),\pi)$ is a pole.
		
		Finally, assume that $n_1=2$. By Observation \ref{obs:0mod4}, $d_1$ is even, hence, $d_1=m/2$ and $\pi = \pi_1[m/2]$. In this case the eigenvalues of the infinitesimal character of $\pi_1$ are $\pm\left(\frac{m/2-1}{2}\right)$ because of Observation \ref{obs:Eigenvalues} so that $\pi_1\cong\pi_f$ for some eigenform $f\in\mathrm{S}_{m/2}(\SL_2(\Z))$. By Proposition \ref{prp:BorcherdsFreitagWeissauerForm}, $\g(\pi)\leq\frac{3m}{4}$.
	\end{proof}
	
	We have seen in the proof that $\g(\pi) = \frac{3m}{4}$ can only occur when $\pi = \pi_f[m/2]$ for an $f\in\mathrm{S}_{m/2}(\SL_2(\Z))$ with $L\left(\frac{1}{2},\pi_f\right) = 0$. This implies Remark \ref{rem}.
	
	\bibliographystyle{amsplain}
	\bibliography{references}

\providecommand{\bysame}{\leavevmode\hbox to3em{\hrulefill}\thinspace}
\providecommand{\MR}{\relax\ifhmode\unskip\space\fi MR }
\providecommand{\MRhref}[2]{%
  \href{http://www.ams.org/mathscinet-getitem?mr=#1}{#2}
}
\providecommand{\href}[2]{#2}
\begin{thebibliography}{10}

\bibitem{Arthur}
J.~Arthur, \emph{The endoscopic classification of representations: Orthogonal
  and symplectic groups}, American Mathematical Society Colloquium
  Publications, vol.~61, American Mathematical Society, Providence, RI, 2013.
  \MR{3135650}

\bibitem{AtobeGan}
H.~Atobe and W.~T. Gan, \emph{On the local {L}anglands correspondence and
  {A}rthur conjecture for even orthogonal groups}, Represent. Theory
  \textbf{21} (2017), 354--415. \MR{3708200}

\bibitem{AGIKMS}
H.~Atobe, W.~T. Gan, A.~Ichino, T.~Kaletha, A.~M\'inguez, and S.~W. Shin,
  \emph{Local {I}ntertwining {R}elations and {C}o-tempered {$A$}-packets of
  {C}lassical {G}roups}, preprint, arXiv: 2410.13504, 2025.

\bibitem{Boecherer}
S.~B\"ocherer, \emph{{\"U}ber den {K}ern der {T}hetaliftung}, Abh. Math. Sem.
  Univ. Hamburg \textbf{60} (1990), 209--223. \MR{1087128}

\bibitem{BFW}
R.~E. Borcherds, E.~Freitag, and R.~Weissauer, \emph{A {S}iegel cusp form of
  degree 12 and weight 12}, J. Reine Angew. Math. \textbf{494} (1998),
  141--153. \MR{1604476}

\bibitem{ChaudouardLaumon1}
P.-H. Chaudouard and G.~Laumon, \emph{Le lemme fondamental pond\'er\'e. {I}.
  {C}onstructions g\'eom\'etriques}, Compos. Math. \textbf{146} (2010), no.~6,
  1416--1506. \MR{2735371}

\bibitem{ChaudouardLaumon2}
\bysame, \emph{Le lemme fondamental pond\'er\'e. {II}. \'enonc\'es
  cohomologiques}, Ann. of Math. (2) \textbf{176} (2012), no.~3, 1647--1781.
  \MR{2979859}

\bibitem{ChenZou}
R.~Chen and J.~Zou, \emph{Arthur's multiplicity formula for even orthogonal and
  unitary groups}, J. Eur. Math. Soc. (JEMS) \textbf{27} (2025), no.~12,
  4769--4843. \MR{4956977}

\bibitem{CL}
G.~Chenevier and J.~Lannes, \emph{Automorphic forms and even unimodular
  lattices}, Results in Mathematics and Related Areas. 3rd Series. A Series of
  Modern Surveys in Mathematics, vol.~69, Springer, Cham, 2019, Kneser
  neighbors of Niemeier lattices, Translated from the French by Reinie Ern\'e.
  \MR{3929692}

\bibitem{Cogdell}
J.~W. Cogdell, \emph{Lectures on {$L$}-functions, converse theorems, and
  functoriality for {${\rm GL}_n$}}, Lectures on automorphic {$L$}-functions,
  Fields Inst. Monogr., vol.~20, Amer. Math. Soc., Providence, RI, 2004,
  pp.~1--96. \MR{2071506}

\bibitem{ConreyFarmer}
J.~B. Conrey and D.~W. Farmer, \emph{Hecke operators and the nonvanishing of
  {$L$}-functions}, Topics in number theory ({U}niversity {P}ark, {PA}, 1997),
  Math. Appl., vol. 467, Kluwer Acad. Publ., Dordrecht, 1999, pp.~143--150.
  \MR{1691315}

\bibitem{FarmerJames}
D.~W. Farmer and K.~James, \emph{The irreducibility of some level 1 {H}ecke
  polynomials}, Math. Comp. \textbf{71} (2002), no.~239, 1263--1270.
  \MR{1898755}

\bibitem{GetzHahn}
J.~R. Getz and H.~Hahn, \emph{An introduction to automorphic
  representations---with a view toward trace formulae}, Graduate Texts in
  Mathematics, vol. 300, Springer, Cham, 2024. \MR{4738301}

\bibitem{GodementJacquet}
R.~Godement and H.~Jacquet, \emph{Zeta functions of simple algebras}, Lecture
  Notes in Mathematics, vol. 260, Springer-Verlag, Berlin-New York, 1972.
  \MR{342495}

\bibitem{Igusa}
J.~Igusa, \emph{Schottky's invariant and quadratic forms}, E. {B}.
  {C}hristoffel ({A}achen/{M}onschau, 1979), Birkh\"auser Verlag, Basel-Boston,
  Mass., 1981, pp.~352--362. \MR{661078}

\bibitem{JacquetShalika}
H.~Jacquet and J.~A. Shalika, \emph{On {E}uler products and the classification
  of automorphic representations. {I}}, Amer. J. Math. \textbf{103} (1981),
  no.~3, 499--558. \MR{618323}

\bibitem{Langlands}
R.~P. Langlands, \emph{Euler products}, Yale Mathematical Monographs, vol.~1,
  Yale University Press, New Haven, Conn.-London, 1971, A James K. Whittemore
  Lecture in Mathematics given at Yale University, 1967. \MR{419366}

\bibitem{MoeglinWaldspurger1}
C.~Moeglin and J.-L. Waldspurger, \emph{Stabilisation de la formule des traces
  tordue. {V}ol. 1}, Progress in Mathematics, vol. 316, Birkh\"auser/Springer,
  Cham, 2016. \MR{3823813}

\bibitem{MoeglinWaldspurger2}
\bysame, \emph{Stabilisation de la formule des traces tordue. {V}ol. 2},
  Progress in Mathematics, vol. 317, Birkh\"auser/Springer, Cham, 2016.
  \MR{3823814}

\bibitem{Rallis}
S.~Rallis, \emph{Langlands' functoriality and the {W}eil representation}, Amer.
  J. Math. \textbf{104} (1982), no.~3, 469--515. \MR{658543}

\bibitem{Salvati}
R.~Salvati~Manni, \emph{Siegel cusp forms of degree {$12k$} and weight
  {$12k$}}, Manuscripta Math. \textbf{101} (2000), no.~2, 267--269.
  \MR{1742244}

\bibitem{Shahidi}
F.~Shahidi, \emph{Eisenstein series and automorphic {$L$}-functions}, American
  Mathematical Society Colloquium Publications, vol.~58, American Mathematical
  Society, Providence, RI, 2010. \MR{2683009}

\bibitem{Shin}
S.~W. Shin, \emph{Weak transfer from classical groups to general linear
  groups}, Essent. Number Theory \textbf{3} (2024), no.~1, 19--62. \MR{4752712}

\bibitem{Taibi}
O.~Ta\"ibi, \emph{Arthur's multiplicity formula for certain inner forms of
  special orthogonal and symplectic groups}, J. Eur. Math. Soc. (JEMS)
  \textbf{21} (2019), no.~3, 839--871. \MR{3908767}

\bibitem{Waldspurger}
J.-L. Waldspurger, \emph{{\`A} propos du lemme fondamental pond\'er\'e{}
  tordu}, Math. Ann. \textbf{343} (2009), no.~1, 103--174. \MR{2448443}

\end{thebibliography}
	
\end{document}